\documentclass[12pt]{amsart}
\usepackage{a4wide}
\usepackage[T1]{fontenc}
\usepackage{amssymb,amsmath,amsthm,latexsym}
\usepackage{mathrsfs}
\usepackage[usenames,dvipsnames]{color}
\usepackage{euscript}
\usepackage{graphicx}
\usepackage{mdwlist}
\usepackage{enumerate}
\usepackage{bm}
\usepackage{mathtools,dsfont,wasysym}
\usepackage{hyperref}
\hypersetup{colorlinks=true,  linkcolor=blue, citecolor=red, urlcolor=cyan}

\newtheorem{theorem}{Theorem}[section]

\newtheorem{proposition}[theorem]{Proposition}

\newtheorem{claim}[theorem]{Claim}

\newtheorem*{mth}{Main Theorem}
\theoremstyle{remark}

\theoremstyle{definition}

\numberwithin{equation}{section}
\makeatother

\newcommand{\R}{\mathbb{R}}
\newcommand{\N}{\mathbb{N}}
\newcommand{\e}{\varepsilon}
\newcommand{\p}{\varphi}
\newcommand{\n}{\left\Vert\cdot\right\Vert}
\newcommand{\nn}[1]{{\left\vert\kern-0.25ex\left\vert\kern-0.25ex\left\vert #1 \right\vert\kern-0.25ex \right\vert\kern-0.25ex \right\vert}}
\newcommand{\ceil}[1]{\left\lceil #1 \right\rceil}
\newcommand{\cut}{\mathord{\upharpoonright}}
\renewcommand{\leq}{\leqslant}
\renewcommand{\geq}{\geqslant}
\DeclareMathOperator{\supp}{supp}

\newcounter{smallromans}

{\end{list}}

\newcommand{\X}{\mathcal{X}}
\newcommand{\Y}{\mathcal{Y}}
\newcommand{\Z}{\mathcal{Z}}
\newcommand{\U}{\mathcal{U}}
\renewcommand\qedsymbol{$\blacksquare$} 

\usepackage{todonotes}

\begin{document}
\title[Smooth norms in dense subspaces of $\ell_p(\Gamma)$ and operator ranges]{Smooth norms in dense subspaces of $\ell_p(\Gamma)$\\and operator ranges}

\author[S.~Dantas]{Sheldon Dantas}
\address[S.~Dantas]{Departament de Matem\`atiques and Institut Universitari de Matem\`atiques i Aplicacions de Castell\'o (IMAC), Universitat Jaume I, Campus del Riu Sec. s/n, 12071 Castell\'o, Spain \newline
\href{https://orcid.org/0000-0001-8117-3760}{ORCID: \texttt{0000-0001-8117-3760}}}
\email{\texttt{dantas@uji.es}}

\author[P.~H\'ajek]{Petr H\'ajek}
\address[P.~H\'ajek]{Department of Mathematics\\Faculty of Electrical Engineering\\Czech Technical University in Prague\\Technick\'a 2, 166 27 Prague 6\\ Czech Republic}
\email{hajek@math.cas.cz}

\author[T.~Russo]{Tommaso Russo}
\address[T.~Russo]{Institute of Mathematics\\ Czech Academy of Sciences\\ \v{Z}itn\'a 25, 115 67 Prague 1\\ Czech Republic\\ and Department of Mathematics\\Faculty of Electrical Engineering\\Czech Technical University in Prague\\Technick\'a 2, 166 27 Prague 6\\ Czech Republic \newline
\href{https://orcid.org/0000-0003-3940-2771}{ORCID: \texttt{0000-0003-3940-2771}}}
\email{russo@math.cas.cz, russotom@fel.cvut.cz}

\thanks{S.~Dantas was supported by the Spanish AEI Project PID2019 - 106529GB - I00 / AEI / 10.13039/501100011033 and the funding received from the Universitat Jaume I through its Research Stay Grants (E-2022-04). P.~H\'ajek was supported in part by OPVVV CAAS CZ.02.1.01/0.0/0.0/16$\_$019/0000778. T.~Russo was supported by the GA\v{C}R project 20-22230L; RVO: 67985840, and by Gruppo Nazionale per l'Analisi Matematica, la Probabilit\`a e le loro Applicazioni (GNAMPA) of Istituto Nazionale di Alta Matematica (INdAM), Italy.}

\date{\today}
\keywords{Smooth norm, local dependence on finitely many coordinates, long sequence spaces, lineability, operator range, Implicit Function Theorem.}
\subjclass[2020]{46B03, 46B20 (primary), and 46B45, 46B26, 47J07, 46T20 (secondary)}

\begin{abstract} For $1\leq p<\infty$, we prove that the dense subspace $\Y_p$ of $\ell_p(\Gamma)$ comprising all elements $y$ such that $y \in \ell_q(\Gamma)$ for some $q \in (0,p)$ admits a $C^{\infty}$-smooth norm which locally depends on finitely many coordinates. Moreover, such a norm can be chosen as to approximate the $\n_p$-norm. This provides examples of dense subspaces of $\ell_p(\Gamma)$ with a smooth norm which have the maximal possible linear dimension and are not obtained as the linear span of a biorthogonal system. Moreover, when $p>1$ or $\Gamma$ is countable, such subspaces additionally contain dense operator ranges; on the other hand, no non-separable operator range in $\ell_1(\Gamma)$ admits a $C^1$-smooth norm.
\end{abstract}
\maketitle

\section{Introduction}
The present paper is a continuation of the research of the authors \cite{DHR, DHR1, HR}, dedicated to the study of smoothness in (incomplete) normed spaces. The main question that we face in this ongoing project is the following: given a Banach space $\X$ and $k\in\N \cup\{\infty,\omega\}$ is there a dense subspace $\Y$ of $\X$ such that $\Y$ admits a $C^k$-smooth norm? (By definition, $C^\omega$-smooth means analytic; however, the case $k=\omega$ will not be considered in this article.) Such a line of research can be traced back at least to the papers \cite{H locally, V} from the early Nineties, where it is proved that every separable Banach space admits a dense subspace with a $C^\infty$-smooth norm. In particular, for a separable normed space $\X$ the existence of a $C^1$-smooth norm does not imply that $\X^*$ is separable, a result that is possibly surprising at first sight. Our goal in \cite{DHR, DHR1} was to push such a theory to the non-separable context and it culminated in the main result of \cite{DHR1} asserting that every Banach space with a fundamental biorthogonal system has a dense subspace with a $C^\infty$-smooth norm (let us also refer to the same paper for a more thorough introduction to the subject).

As it turns out, in most of the above results, the dense subspace $\Y$ of $\X$ is the linear span of a certain biorthogonal system in $\X$ (the unique exception being \cite[Theorem 3.1]{DHR} where an analytic norm is constructed in the dense subspace $\ell_\infty^F={\rm span}\{\bm{1}_A\colon A\subseteq \N\}$ of $\ell_\infty$). In particular, when $\X$ is separable, the subspace $\Y$ has countable dimension (namely, it is the linear span of a countable set). In this paper we focus on the classical (long) sequence spaces and we show that it is possible to go beyond this limitation; in particular, we build $C^\infty$-smooth norms on dense subspaces that are `large' in a sense that we specify below. More precisely, the following is our main result (for the necessary notation, we refer the reader to Subsection \ref{subsect:notation} below).

\begin{mth} Let $1\leq p <\infty$ and $\Gamma$ be any infinite set. Then 
\begin{equation*}
    \Y_p\coloneqq \Big\{y\in \ell_p(\Gamma)\colon \|y\|_q < \infty \text{ for some } q\in(0,p) \Big\}= \bigcup_{0<q<p} \ell_q(\Gamma)
\end{equation*}
is a dense subspace of $\ell_p(\Gamma)$ which admits a $C^{\infty}$-smooth and LFC norm that approximates the $\n_p$-norm.
\end{mth}

Plainly, when $p$ is an even integer, the canonical norm of $\ell_p(\Gamma)$ is $C^\infty$-smooth, hence the $C^\infty$-smoothness part of the theorem is obvious. On the other hand, if $p\notin 2\N$, $\ell_p(\Gamma)$ does not have any $C^{\ceil p}$-smooth norm ($\ceil p$ denotes the ceiling of $p$) \cite[p.~295]{HJ}. Moreover, no $\ell_p(\Gamma)$ has an LFC norm \cite{PWZ}, so in order to obtain an LFC norm in the main theorem it is indeed necessary to pass to the subspace $\Y_p$. Finally, recall that $c_0(\Gamma)$ has a $C^\infty$-smooth and LFC norm \cite[p.~284]{HJ}, for which reason we do not consider $c_0(\Gamma)$ in our result. \smallskip

Let us now discuss the novelty of the result. First of all, for every $1\leq p<\infty$, the dense subspace $\Y_p$ of $\ell_p(\Gamma)$ has the same linear dimension of $\ell_p(\Gamma)$, hence it is as large as possible in the linear sense. In particular, in the case of $\ell_p$, \emph{i.e.}, when $\Gamma$ is countable, we obtain a dense subspace of dimension continuum. This is in sharp contrast with the results in \cite{DHR, DHR1, H locally, V}, where the dense subspaces had dimension equal to the density character of the Banach space $\X$. Comparing this result with \cite{H locally}, it seems conceivable to conjecture that for every separable Banach space $\X$ there is a dense subspace $\Y$ of dimension continuum and with a $C^\infty$-smooth norm. We leave the validity of such a conjecture as an open problem.

Moreover, if $\nn\cdot$ is a norm on $\ell_p(\Gamma)$ that coincides with the $C^\infty$-smooth and LFC one on $\Y_p$, then it is standard to verify (see, \emph{e.g.}, the proof of \cite[Corollary 3.5]{DHR1}) that the norm $\nn\cdot$ is $C^\infty$-smooth and LFC at every point of $\Y_p$ (as a function on $\ell_p(\Gamma)$). In lineability terms, this assertion can be restated as stating that the set of points in $\ell_p(\Gamma)$ where the norm $\nn\cdot$ is $C^\infty$-smooth and LFC is maximal densely lineable in $\ell_p(\Gamma)$. Let us refer to \cite{ABPS, BPS} and the references therein for information on lineability in Banach spaces and for the relevant definitions.

Finally, there is a second sense in which the subspace $\Y_p$ can be considered to be `large', which is connected to the notion of operator ranges, \cite{OpRange1, OpRange2, OpRange3}. Recall that a normed space $\Y$ is an \emph{operator range} if there are a Banach space $\Z$ and a surjective bounded linear operator $T\colon \Z\to \Y$; in other words, $\Y$ is the linear image of a Banach space. Notice that operator ranges bear a certain form of completeness, since, for instance, they satisfy the Baire Category Theorem, even if in a finer linear topology. When $p>1$, it is clear that $\Y_p$ contains a dense operator range, since the Banach space $\ell_1(\Gamma)$ injects in $\Y_p$. Hence, when $p>1$, our theorem also implies the existence of a dense operator range in $\ell_p(\Gamma)$ that admits a $C^\infty$-smooth and LFC norm. On the other hand, the situation is different when $p=1$ (and $\Gamma$ is uncountable): indeed, it is a folklore result, essentially due to Rosenthal \cite{Rosenthal}, that every non-separable operator range in $\ell_1(\Gamma)$ contains an isomorphic copy of $\ell_1(\omega_1)$, hence it admits no $C^1$-smooth norm (see Proposition \ref{Proposition:l1-omega1}). This result is extremely relevant to the topic of the paper, since it is one of the few instances where an incomplete normed space is proved not to admit any $C^1$-smooth norm. Moreover it is also the first occurrence where $\ell_1(\Gamma)$ behaves worse than $\ell_p(\Gamma)$ ($1<p<\infty$) for what concerns the existence of smooth norms on dense subspaces. Such behaviour was to be expected, but it was not present in the literature so far; quite surprisingly, there are in fact instances of the opposite situation (see the discussion concerning Theorem A(iii) in \cite{DHR}).

\subsection{Definitions and notation} \label{subsect:notation}
All the spaces that we consider in the paper are real normed spaces. If $\X$ is a normed space, then the norm $\n$ of $\X$ is said to be \emph{$C^k$-smooth} if its $k$-th Fr\'echet derivative exists and it is continuous at every point of $\X \setminus \{0\}$. When this holds for every $k \in \N$, the norm is \emph{$C^\infty$-smooth}. The norm $\n$ \emph{locally depends on finitely many coordinates} (is LFC, for short) on $\X$ if for each $x\in \X\setminus \{0\}$ there exist an open neighbourhood $\U$ of $x$, functionals $\p_1,\dots,\p_k \in\X^*$, and a $C^\infty$-smooth function $G\colon \R^k\to\R$ such that
\begin{equation*}
    \|y\|= G\big( \langle\p_1,y\rangle,\dots, \langle\p_k,y\rangle \big) \qquad \text{for every } y\in\U.
\end{equation*}
We say that a norm $\n$ on $\X$ \emph{can be approximated} by norms with a certain property $P$ if, for every $\e>0$, there is an equivalent norm $\nn\cdot$ on $\X$ with property $P$ and such that $(1-\e)\nn\cdot\leq \n\leq (1+\e)\nn\cdot$. For further information concerning smoothness we refer to the monographs \cite{DGZ, HJ}.

If $x\colon\Gamma\to \R$ and $p\in(0,\infty)$, we write as usual
\begin{equation*}
    \|x\|_p\coloneqq \left( \sum_{\gamma\in \Gamma} |x(\gamma)|^p\right)^{1/p} \qquad \text{and}\qquad \ell_p(\Gamma)\coloneqq\Big\{x\colon\Gamma \to\R\colon \|x\|_p<\infty \Big\}.
\end{equation*}
When $p\geq1$, $\ell_p(\Gamma)$ is obviously a Banach space with the norm $\n_p$, while for $p\in(0,1)$ it is only a quasi-Banach space and $\n_p$ is a quasi-norm. When $0<q<p<\infty$, we sometimes write $\ell_q(\Gamma) \subseteq \ell_p(\Gamma)$, by which we mean the inclusion of the corresponding underlying vector spaces (which is also a continuous injection of (quasi-)Banach spaces).

We write $\N$ for the set of positive natural numbers and $\N_0\coloneqq \N\cup\{0\}$. We write $\omega_1$ for the smallest uncountable ordinal. We denote by $|A|$ the cardinality of a set $A$. Given a set $\Gamma$, we use the standard set-theoretic notation $[\Gamma]^{<\omega} \coloneqq \{A\subseteq \Gamma\colon |A|<\infty\}$.

\section{Proof of the Main Theorem}
In this section we provide the proof of our main theorem. The argument is inspired by the proof of \cite[Theorem 4.1]{DHR}; in particular, the formula for the norm and the scheme of the argument are essentially the same. Nevertheless, the similarity between the two proofs is more formal than substantial due to the crucial part of the proof (Claim \ref{Claim:jump-ngh} below, that corresponds to \cite[Claim 4.2]{DHR}) which is rather different in the two papers: indeed, in the present paper we require some new ingredients, including combinatorial ones and finer estimates. 

In what follows, we occasionally write $a_k \sim b_k$ to mean that the scalar sequences $(a_k)_{k=1}^\infty$ and $(b_k)_{k=1}^\infty$ are asymptotic (\emph{i.e.}, $a_k/b_k\to1$ as $k\to \infty$).

\begin{proof}[Proof of the Main Theorem] Let us start by choosing some parameters. Let $(\delta_k)_{k=0}^{\infty} \subseteq\R$ be a decreasing sequence with  
\begin{equation} \label{sequence:delta-k}
	\delta_k \sim \frac{1}{\log k} \qquad \text{as } k\to\infty.
\end{equation}
Next, let $(\theta_k)_{k=0}^{\infty} \subseteq \R$ be another decreasing sequence such that 
\begin{equation} \label{sequence:theta-k}
    \theta_k \searrow 0 \ \ \ \mbox{and} \ \ \ \frac{1 + \delta_{k+1}}{1 + \delta_k} < 1 - 2 \theta_{k+1} \ \mbox{for every} \ k \geq 0.
\end{equation}
Additionally, fixed $\e>0$, the sequences are chosen so that
\begin{equation} \label{epsilon} 
    \frac{1+\theta_1}{1-\theta_1}\cdot (1+\delta_1)^2\leq 1+\e.
\end{equation}

Next, we fix some notation we will be using throughout the proof. For $A\in [\Gamma]^{<\omega}$, we identify $\ell_p(A)$ with the finite-dimensional subspace of $\ell_p(\Gamma)$ comprising all vectors with support contained in $A$. Moreover, when $x\in\ell_p(\Gamma)$, we define $Ax\in\ell_p(A)$ as
\begin{equation*}
    (Ax)(\gamma)\coloneqq 
    \begin{cases} x(\gamma) & \text{if}\ \gamma\in A \\
    0 & \text{if}\ \gamma\notin A.
    \end{cases}
\end{equation*}
In other words, we also denote by $A$ the canonical projection from $\ell_p(\Gamma)$ onto $\ell_p(A)$. Since $\ell_p(A)$ is finite-dimensional, $C^\infty$-smooth norms are dense in $\ell_p(A)$, hence there exists a $C^\infty$-smooth norm $\n_{s,A}$ on $\ell_p(A)$ that $\frac{1}{1 + \theta_{|A|}}$-approximates the $\n_p$-norm ($\n_{s,A}$ is also trivially LFC). In particular, for every $y\in\Y$ we have
\begin{equation} \label{first-inequality}
    \frac{1}{1 + \theta_{|A|}} \|Ay\|_p \leq \|Ay\|_{s,A} \leq \|Ay\|_p.
\end{equation}
Finally, in several computations it will be more convenient to use, instead of $\n_p$, the following auxiliary equivalent norm on $\Y$
\begin{equation*}
    \nu(x) \coloneqq \sup_{|A|<\infty} (1 + \delta_{|A|})^2 \|Ax\|_p \ \ \ (x \in \Y). 
\end{equation*}
The norm $\nu$ approximates the norm $\n_p$ on $\Y$; more precisely, 
\begin{equation} \label{second-inequality}
    \n_p \leq \nu \leq (1+\delta_1)^2 \n_p
\end{equation}	

We now come to the crucial part of the proof, where we prove the following `strong maximum' result.

\begin{claim} \label{Claim:jump-ngh} Let $x \in \Y$ be such that $\nu(x) \leq 1$. Then there exist an open neighbourhood $\mathcal{O}_x$ of $x$ and a finite collection of subsets $\mathfrak{F}_x \subseteq [\Gamma]^{<\omega}$ such that, for each $y \in \mathcal{O}_x$ and each $A \in [\Gamma]^{<\omega} \setminus \mathfrak{F}_x$, we have 
\begin{equation} \label{jump-ngh}
    (1+\delta_{|A|})^2 \|Ay\|_{s,A} \leq 1-\theta_{|A|}. 
\end{equation}
\end{claim}		

\begin{proof}[Proof of Claim \ref{Claim:jump-ngh}] \renewcommand\qedsymbol{$\square$}
In fact, we prove a stronger estimate that only involves the point $x$. Indeed, we show that there exist a finite collection of subsets $\mathfrak{F}_x \subseteq [\Gamma]^{<\omega}$ and $k_0 \in \N$ such that, for every $A \in [\Gamma]^{<\omega} \setminus \mathfrak{F}_x$, we have 
\begin{equation} \label{jump}
    (1+\delta_{|A|})^2 \|Ax\|_p \leq 
    \begin{cases} \displaystyle
    1-2\theta_{|A|}, & \mbox{if}  \ |A|\leq k_0,\\  \displaystyle 1-2\theta_{k_0+1}, & \mbox{if} \ |A|>k_0.
    \end{cases}
\end{equation}

Let us show first that \eqref{jump} is indeed stronger than \eqref{jump-ngh}. Suppose that we have proved \eqref{jump}. Define the following open neighbourhood of $x$:
\begin{equation*}
    \mathcal{O}_x\coloneqq \Big\{y\in\Y\colon \nu(y-x)< \theta_{k_0+1} \Big\}.
\end{equation*}
Then, for every $A \in [\Gamma]^{<\omega} \setminus \mathfrak{F}_x$ and $y \in \mathcal{O}_x$, we have that
\begin{eqnarray*}
    (1 + \delta_{|A|})^2 \|Ay\|_{s,A} &\stackrel{\eqref{first-inequality}}{\leq}& ( 1 + \delta_{|A|})^2 \|Ay\|_p\\
    &\leq& (1 + \delta_{|A|})^2 \|Ax\|_p + (1 + \delta_{|A|})^2 \|A(y-x)\|_p \\
    &\leq& (1 + \delta_{|A|})^2 \|Ax\|_p + \nu(y - x)< (1 + \delta_{|A|})^2 \|Ax\|_p + \theta_{k_0+1}.	
\end{eqnarray*} 
If $|A|\leq k_0$, we use \eqref{jump} and continue from the above inequalities
\begin{equation*}
    (1+\delta_{|A|})^2 \|Ay\|_{s,A} \stackrel{\eqref{jump}}{\leq}  1 - 2 \theta_{|A|} + \theta_{k_0+1}\leq 1 - \theta_{|A|}. 
\end{equation*}
Similarly, if $|A| > k_0$, 
\begin{equation*}
    (1+\delta_{|A|})^2 \|Ay\|_{s,A}\leq 1-\theta_{k_0+1}\leq 1-\theta_{|A|}. 
\end{equation*}

Therefore to prove Claim \ref{Claim:jump-ngh} all we need to do is to prove \eqref{jump}. In order to do so, let us consider the sequence $(\phi_k(x))_{k=1}^{\infty}$ defined by 
\begin{equation*}
    \phi_k(x)\coloneqq (1+\delta_k) \sup_{|A|=k} \|Ax\|_p \qquad (k\in\N).
\end{equation*}
Note that the supremum above is actually attained. Indeed, if $(\gamma_j)_{j=1}^{\infty} \subseteq\Gamma$ is an injective sequence such that $\supp(x) \subseteq \{\gamma_j\}_{j=1}^{\infty}$ and $\big(|x(\gamma_j)| \big)_{j=1}^{\infty}$ is non-increasing, it is clear that
\begin{equation*}
    \phi_k(x)=(1+\delta_k)\cdot \left(\sum_{j=1}^k |x(\gamma_j)|^p \right)^{1/p} \qquad \text{for every}\ k\in\N.
\end{equation*}

\begin{claim} \label{Claim:phi-n-is-decreasing-eventually} There exists $k_0\in\N$ such that $\phi_{k+1}(x)\leq \phi_k(x)$ for every $k \geq k_0$. 
\end{claim}

\begin{proof}[Proof of Claim \ref{Claim:phi-n-is-decreasing-eventually}] \renewcommand\qedsymbol{$\square$}
Clearly, $\phi_{k+1}(x)\leq \phi_k(x)$ is equivalent to
\begin{equation*}
    (1+\delta_{k+1})^p\cdot \left(\sum_{j=1}^k |x(\gamma_j)|^p+ |x(\gamma_{k+1})|^p \right) \leq(1+ \delta_k)^p \cdot\sum_{j=1}^k |x(\gamma_j)|^p
\end{equation*}
which is in turn equivalent to
\begin{equation} \label{claim-inequality}
    |x(\gamma_{k+1})|^p\leq \frac{(1+\delta_k)^p - (1+\delta_{k+1})^p}{(1+\delta_{k+1})^p} \cdot \sum_{j=1}^k |x(\gamma_j)|^p.
\end{equation}
Thus, it is enough to check that \eqref{claim-inequality} is true for all large enough $k$. By using first-order Taylor expansions and \eqref{sequence:delta-k} we readily get
\begin{equation*}
    (1+\delta_k)^p - (1+\delta_{k+1})^p \sim\frac{p}{k (\log k)^2}.
\end{equation*}
As the sequence $\left(\sum_{j=1}^k |x(\gamma_j)|^p \right)_{n=1}^{\infty}$ is non-decreasing and bounded, for some constant $C > 0$ we thus have that 
\begin{equation} \label{phi-n-estimation1}
    \frac{(1+\delta_k)^p - (1+\delta_{k+1})^p}{ (1+\delta_{k+1})^p} \cdot\sum_{j=1}^k |x(\gamma_j)|^p \sim C \cdot \frac{1}{k(\log k)^2}.
\end{equation}

We now estimate the left-hand side of \eqref{claim-inequality}. Since $x \in \Y$, there exists $q < p$ such that $x \in \ell_q(\Gamma)$. By definition of the sequence $(\gamma_j)_{j=1}^\infty$, for every $k \in \N$, we have that\footnote{This is essentially the Chebyshev--Markov inequality.}
\begin{equation*}
    \|x\|_q^q \geq \sum_{j=1}^k |x(\gamma_j)|^q \geq k |x(\gamma_k)|^q.
\end{equation*}
In other words, there exists $\tilde{C} > 0$ such that 
\begin{equation} \label{phi-n-estimation2}
    |x(\gamma_{k+1})|^p \leq \tilde{C} \cdot \frac{1}{k^{p/q}}
\end{equation}
Comparing \eqref{phi-n-estimation1} and \eqref{phi-n-estimation2} and recalling that $\frac{p}{q}>1$, it is clear that \eqref{claim-inequality} is true for $k$ sufficiently large, which proves Claim \ref{Claim:phi-n-is-decreasing-eventually} as desired.
\end{proof} 

We now return to the proof of Claim \ref{Claim:jump-ngh}. Note first that $(1+\delta_k)\phi_k(x)\leq\nu(x)$, for every $k\in\N$. Let $k_0\in\N$ as in Claim \ref{Claim:phi-n-is-decreasing-eventually}. Then, for every set $A\in [\Gamma]^{<\omega}$ with $|A|>k_0$, we have 
\begin{eqnarray*}
    (1+\delta_{|A|})^2 \|Ax\|_p &\leq& (1 + \delta_{k_0+1})(1 + \delta_{|A|}) \|Ax\|_p \\
    &\leq& \frac{1 + \delta_{k_0+1}}{1 +\delta_{k_0}} \cdot (1 + \delta_{k_0}) \phi_{|A|}(x)\\
    &\leq& \frac{1 + \delta_{k_0+1}}{1 +\delta_{k_0}} \cdot (1 + \delta_{k_0}) \phi_{k_0}(x) \\
    &\leq& \frac{1 + \delta_{k_0+1}}{1 +\delta_{k_0}} \cdot \nu(x) \leq \frac{1 + \delta_{k_0+1}}{1 +\delta_{k_0}} \stackrel{\eqref{sequence:theta-k}}{<} 1 - 2 \theta_{k_0+1}. 
\end{eqnarray*} 

It remains to prove that there exists a finite subset $\mathfrak{F}_x$ of $[\Gamma]^{<\omega}$ such that, for each $A \in [\Gamma]^{<\omega} \setminus \mathfrak{F}_x$ with $|A| \leq k_0$, we have that $(1 + \delta_{|A|})^2 \|Ax\|_p \leq 1 - 2 \theta_{|A|}$. Suppose that this is not the case. Then, there exists a sequence of mutually distinct sets $(A_j)_{j=1}^{\infty} \subseteq \Gamma$ with $|A_j|\leq k_0$ for every $j\in\N$ and such that 
\begin{equation} \label{claim-inequality-1}
    (1 + \delta_{|A_j|})^2 \|A_j x\|_p > 1 - 2 \theta_{|A_j|}.
\end{equation}
Up to passing to a subsequence, we can assume that there is $k\in\N$ such that $|A_j|=k$ for every $j\in\N$. Hence, since all the sets $A_j$ have the same cardinality, we can apply the Delta System Lemma for countable families (see, for instance, \cite[p.~167]{K}). Therefore, there exist a subsequence of $(A_j)_{j=1}^{\infty}$, still denoted by $(A_j)_{j=1}^{\infty}$, and a set $\Delta \in [\Gamma]^{<\omega}$ (possibly empty) such that $A_i\cap A_j=\Delta$ for every $i\neq j$. Notice that $|\Delta| \leq k-1$ because the sets $A_j$ are mutually distinct. Now, since the elements of the sequence $(A_j \setminus \Delta)_{j=1}^{\infty}$ are disjoint and $x \in \ell_p(\Gamma)$, we have that
\begin{equation*}
    \|A_j x\|_p^p = \|\Delta x\|_p^p + \|(A_j\setminus \Delta)x\|_p^p \to \|\Delta x\|_p^p \qquad \text{as} \ j\to\infty.
\end{equation*} 
Thus, taking the limit when $j \to \infty$ in \eqref{claim-inequality-1}, we get that 
\begin{equation} \label{claim-inequality-2} 
    (1+\delta_k)^2 \|\Delta x\|_p \geq 1-2\theta_k 
\end{equation}
which will yield a contradiction. Indeed, since $(1+\delta_{|\Delta|})^2 \|\Delta(x)\|_p \leq\nu(x) \leq 1$ and $|\Delta| \leq k-1$, \eqref{claim-inequality-2} yields
\begin{equation*}
    1-2\theta_k\leq \frac{(1+\delta_k)^2}{(1+ \delta_{|\Delta|})^2} \cdot (1+\delta_{|\Delta|})^2 \|\Delta x\|_p\leq 
    \left( \frac{1 + \delta_k}{1 + \delta_{k-1}} \right)^2 \leq \frac{1 + \delta_k}{1 + \delta_{k-1}} \stackrel{\eqref{sequence:theta-k}}{<} 1 - 2 \theta_k
\end{equation*}
which is an absurd. This concludes the proof of Claim \ref{Claim:jump-ngh}. 
\end{proof}	
	
From this point on, we just have to glue together the ingredients in the standard way. Even if the argument is the same as in \cite[Theorem 4.1]{DHR}, we give the details for the sake of being self-contained.

For every $n\in \N_0$, let $\rho_n\colon \R\to [0,\infty)$ be a $C^{\infty}$-smooth, even, and convex function such that $\rho_n\equiv0$ on $[0,1- \theta_n^2]$ and $\rho_n(1)=1$. Then, for every $n\in\N_0$, we have that $\rho_n(t) \leq 1$ if and only if $|t|\leq1$. Define $\Psi\colon \Y\to [0,\infty]$ by
\begin{equation}\label{eq: def Psi}
    \Psi(x)\coloneqq\sum_{|A|<\infty} \rho_{|A|} \Big((1+\delta_{|A|})^2 \cdot (1+\theta_{|A|})\cdot \|Ax\|_{s,A} \Big) \qquad (x \in \Y).
\end{equation}
Let $x \in\Y$ with $\nu(x) \leq 1$ and take $\mathcal{O}_x$ and $\mathfrak{F}_x$ as in Claim \ref{Claim:jump-ngh}. If $y\in \mathcal{O}_x$ and $A \in [\Gamma]^{<\omega} \setminus \mathfrak{F}_x$, we have that 
\begin{equation*}
    (1 + \delta_{|A|})^2 (1 + \theta_{|A|}) \|Ay\|_{s,A} \stackrel{\eqref{jump-ngh}}{\leq} (1 - \theta_{|A|})(1 + \theta_{|A|}) = 1 - \theta_{|A|}^2
\end{equation*}
which implies that $\rho_{|A|} \Big((1+\delta_{|A|})^2 \cdot (1+\theta_{|A|})\cdot \|Ay\|_{s,A} \Big) = 0$. Hence, on the set $\mathcal{O}_x$, only the finitely many summands with $A\in \mathfrak{F}_x$ are different from $0$. Also, note that each summand in \eqref{eq: def Psi} is $C^\infty$-smooth on $\Y$ (each summand vanishes in a neighbourhood of $0$, so $\Psi$ is also differentiable there). Therefore, $\Psi$ is (real-valued and) $C^\infty$-smooth and LFC on the open set $\mathcal{O}$ defined by 
\begin{equation*}
    \mathcal{O}\coloneqq \bigcup \Big\{\mathcal{O}_x\colon x \in\Y, \ \nu(x)\leq1 \Big\}. 
\end{equation*}

Next, we note that the convex and symmetric set $\{\Psi<1\}$ is contained in $\{\nu\leq 1\} \subseteq \mathcal{O}$ (hence it is also open, as $\Psi$ is continuous on $\mathcal{O}$). Indeed, if $\Psi(x)<1$, then $\rho_{|A|} \Big((1+\delta_{|A|})^2 \cdot (1+\theta_{|A|})\cdot \|Ax\|_{s,A} \Big) \leq 1$ for every $A \in [\Gamma]^{<\omega}$. So the properties of the functions $\rho_n$ give 
\begin{equation*}
    (1 + \delta_{|A|})^2 \|Ax\|_p \stackrel{\eqref{first-inequality}}{\leq} (1 + \delta_{|A|})^2(1 + \theta_{|A|}) \|Ax\|_{s,A} \leq 1
\end{equation*}
and $\nu(x)\leq 1$. Moreover, the set $\{\Psi\leq 1-\theta_1\}$ is closed in $\Y$ by the lower semi-continuity of $\Psi$. Hence, a standard consequence of the Implicit Function Theorem (see \cite[Lemma 2.5]{DHR} or \cite[Chapter 5, Lemma 23]{HJ}) implies that the Minkowski functional $\nn\cdot$ of $\{\Psi\leq 1-\theta_1\}$ is an equivalent $C^\infty$-smooth and LFC norm on $\Y$.

It remains to check that $\nn\cdot$ approximates $\n_p$. For this aim, we first show that $\Psi(x)=0$ whenever $\nu(x)\leq \frac{1-\theta_1}{1+\theta_1}$. Indeed, if $x\in\Y$ satisfies $\nu(x)\leq \frac{1-\theta_1}{1+\theta_1}$, then for every $A \in [\Gamma]^{<\omega}$, we have that 
\begin{eqnarray*}
    (1 + \delta_{|A|})^2 (1 + \theta_{|A|}) \|Ax\|_{s,A} &\stackrel{\eqref{first-inequality}}{\leq}& (1 + \delta_{|A|})^2(1 + \theta_{|A|}) \|Ax\|_p \\
    &\leq& \frac{1 - \theta_1}{1 + \theta_1} \cdot (1 + \theta_{|A|}) \leq 1 - \theta_{|A|}^2.
\end{eqnarray*}
Hence $\Psi(x) = 0$. Combining this inclusion with the inclusion $\{\Psi\leq1\} \subseteq \{\nu\leq1\}$, which was proved above, we have that 
\begin{equation*}
    \nu\leq \nn\cdot\leq \frac{1+\theta_1}{1-\theta_1} \nu.
\end{equation*}
Together with \eqref{second-inequality} we finally reach the conclusion that 
\begin{equation*}
    \n_p\leq \nn\cdot\leq \frac{1+\theta_1}{1- \theta_1} \cdot (1+\delta_1)^2 \n_p \stackrel{\eqref{epsilon}}{\leq}(1+\e) \n_p.
\end{equation*}
\end{proof}

\section{Operator ranges in \texorpdfstring{$\ell_1(\Gamma)$}{l₁(Gamma)}}
In this short section, we discuss the problem of whether $\ell_p(\Gamma)$ (for $1\leq p<\infty$) contains a dense operator range with a $C^\infty$-smooth and LFC norm. As we mentioned already in the Introduction, operator ranges bear a certain form of completeness, that is not shared by all normed spaces (for example, recall the standard fact that there is no complete norm on a normed space of dimension less than continuum). For this reason, building a smooth norm on an operator range is more complicated than building one on a general normed space.

When $p>1$ we have observed before that $\Y_p$ contains a dense operator range, since there is a continuous linear injection of $\ell_1(\Gamma)$ in $\Y_p$. Therefore, $\ell_p(\Gamma)$ contains a dense operator range with a $C^\infty$-smooth and LFC norm (by the Main Theorem). When $p=1$ the situation is different and the result depends on the cardinality of $\Gamma$. If $\Gamma$ is countable, then $\Y_1$ still contains a dense operator range. In fact, it is sufficient to find a continuous linear injection of $\ell_\infty$ into $\Y_1$, which is simply given by the map $T\colon \ell_\infty \to \ell_{1/2}$ defined by
\begin{equation*}
    (x(j))_{j=1}^\infty \mapsto \left(2^{-j}\cdot x(j)\right)_{j=1}^\infty.
\end{equation*}
Hence, $\ell_1$ also contains a dense operator range with a $C^\infty$-smooth and LFC norm. This ceases to be true for uncountable $\Gamma$, as the next result shows. It is essentially due to Rosenthal \cite{Rosenthal} and, in a slightly weaker form, it can also be found in \cite[Lemma 3.8]{HK}. Yet, a direct proof is so short that we give it here for the sake of completeness.

\begin{proposition} \label{Proposition:l1-omega1} Let $\Y$ be a non-separable operator range in $\ell_1(\Gamma)$. Then $\Y$ contains an isomorphic copy of $\ell_1(\omega_1)$; in particular, $\Y$ admits no Fr\'echet smooth norm.
\end{proposition}

\begin{proof} Let $\Z$ be a Banach space and $T\colon \Z\to \ell_1(\Gamma)$ be a bounded linear operator such that $\Y=T[\Z]$. Then $\overline{\Y}$ is a non-separable subspace of $\ell_1(\Gamma)$, hence it contains an isomorphic copy of $\ell_1(\omega_1)$, \cite[(5) on p.185]{Kothe}. Let $(y_\alpha)_{\alpha<\omega_1}$ be equivalent to the canonical basis of $\ell_1(\omega_1)$, fix $\e>0$, and take $z_\alpha\in \Z$ with $\|y_\alpha - Tz_\alpha\|<\e$. If $\e>0$ is sufficiently small, the sequence $(Tz_\alpha)_{\alpha<\omega_1}$ is also equivalent to the canonical basis of $\ell_1(\omega_1)$ (see, \emph{e.g.}, \cite[Lemma 5.2]{HKR}, or \cite[Example 30.12]{Jameson}). Moreover, up to passing to an uncountable subset of $\omega_1$ and relabelling, we can assume that $(z_\alpha)_{\alpha<\omega_1}$ is a bounded set. Consequently, the linear map $S\colon {\rm span}\{Tz_\alpha\} _{\alpha<\omega_1} \to {\rm span}\{z_\alpha\} _{\alpha<\omega_1}$ defined by $Tz_\alpha \mapsto z_\alpha$ ($\alpha<\omega_1$) is bounded. So $T\cut_{\overline{\rm span} \{z_\alpha\} _{\alpha<\omega_1}}$ is an isomorphic embedding, and we are done.
\end{proof}

\subsection*{Acknowledgements} The third-named author is indebted to Bence Horv\'ath for several conversations concerning \cite{Rosenthal}. Besides, the authors would like to thank Rub\'en Medina for fruitful discussions on previous versions of the paper. During the preparation of this manuscript, the first-named author was visiting the Department of Mathematics of the Faculty of Electrical Engineering of the Czech Technical University in Prague and he is grateful for all the hospitality he had received there.


\end{document}